\documentclass[12pt]{amsart}
\usepackage{amssymb, amstext, amscd, amsmath}
\usepackage{fullpage}

\theoremstyle{plain}
\newtheorem{thm}{Theorem}[section]
\newtheorem{cor}[thm]{Corollary}
\newtheorem{prop}[thm]{Proposition}
\newtheorem{lem}[thm]{Lemma}
\newtheorem{conj}[thm]{Conjecture}

\theoremstyle{definition}

\theoremstyle{remark}
\newtheorem{rem}[thm]{Remark}
\newtheorem{nota}[thm]{Notation}

\begin{document}

\title{Invariant and hyperinvariant subspaces for amenable operators}
\title{Invariant and hyperinvariant subspaces for amenable operators}

\author[L. Y. Shi]{Luo Yi Shi}

\address{Department of Mathematics\\Tianjin Polytechnic University\\Tianjin 300160\\P.R. CHINA}

\email{sluoyi@yahoo.cn}

\author[Y. J. Wu]{YU Jing Wu}

\address{Tianjin Vocational Institute \\Tianjin 300160\\P.R. CHINA}

\email{wuyujing111@yahoo.cn}

\author[Y.Q. Ji]{You Qing Ji}

\address{Department of Mathematics\\Jilin University\\Changchun 130012\\P.R. CHINA}

\email{jiyq@jlu.edu.cn}

\thanks{Supported by NCET(040296), NNSF of China(10971079) and
        the Specialized Research Fund for the Doctoral Program
        of Higher Education(20050183002) }

\date{7.04. 2010}

 \subjclass[2000]{47C05 (46H35 47A65 47A66 47B15)}

\keywords{Amenable; Invariant subspaces; Hyperinvariant subspaces; Reduction property}

\begin{abstract}There has been a long-standing conjecture in
Banach algebra that every amenable operator is similar to a normal operator. In this paper, we study the
structure of amenable operators on Hilbert spaces. At first, we show that the conjecture is equivalent to  every
non-scalar amenable operator has a non-trivial hyperinvariant subspace and equivalent to every amenable operator
is similar to a reducible operator and has a non-trivial invariant subspace; and then, we give two
decompositions for amenable operators, which supporting the conjecture.

\end{abstract}

\maketitle

\section{Introduction}

Throughout this paper,  $\mathfrak{H}$ denotes a complex separable infinite-dimension Hilbert space and
$\mathfrak{B}(\mathfrak{H})$ denotes the bounded linear operators on $\mathfrak{H}$. For an algebra
$\mathfrak{A}$ in $\mathfrak{B}(\mathfrak{H})$, we write $\mathfrak{A}'$ for the commutant of $\mathfrak{A}$
(i.e., $\mathfrak{A}'=\{B\in\mathfrak{B}(\mathfrak{H}), BA=AB ~~\textup{for all} A\in \mathfrak{A}\})$ and
$\mathfrak{A}''$ for the double commutant of  $\mathfrak{A}$ (i.e., $\mathfrak{A}''=(\mathfrak{A}')'$). We also
write Lat$\mathfrak{A}$ for the collection of those subspaces which are invariant for every operator in
$\mathfrak{A}$. We say $\mathfrak{A}$ is {\it completely reducible} if for every subspace $M$ in
Lat$\mathfrak{A}$ there exists $N$ in Lat$\mathfrak{A}$ such that $\mathfrak{H}=M\dot{+}N$ (i.e., $M\cap
N=\{0\}$ and $\mathfrak{H}$ is the algebraic direct sum of $M$ and $N$); $\mathfrak{A}$ is {\it reducible} if
for every subspace $M$ in Lat$\mathfrak{A}$ we have $M^\bot$ (the orthogonal complement of $M$) in
Lat$\mathfrak{A}$; $\mathfrak{A}$ is {\it transitive} if Lat$\mathfrak{A}=\{\{0\}, \mathfrak{H}\}$. If
$T\in\mathfrak{B}(\mathfrak{H})$, we say that a subspace $M$ of $\mathfrak{H}$ is a {\it hyperinvariant
subspace} for $T$ if $M$ is invariant under each operator in $\mathfrak{A}_T'$; $M$ is a {\it reducible
subspace} for $T$ if $M, M^\bot\in$ Lat$T$.

The concept of amenable Banach algebras was first introduced by B. E. Johnson in \cite{BE1972}. Suppose that
$\mathfrak{A}$ is a Banach algebra. A {\it Banach $\mathfrak{A}$-bimodule} is a Banach space $X$ that is also an
algebra $\mathfrak{A}$-bimodule for which there exists a constant  $K>0$ such that $||a\cdot x||\leq K
||a||||x||$ and $||x\cdot a||\leq K ||a||||x||$ for all $a\in \mathfrak{A}$ and $x\in X$. We note that $X^*$,
the dual of $X$, is a Banach $\mathfrak{A}$-bimodule with respect to the dual actions
$$ [a \cdot f](x)=f(x \cdot a), [f \cdot a](x)=f(a \cdot x), a\in \mathfrak{A}, x\in X,
f\in X^*.$$
 Such a Banach $\mathfrak{A}$-bimodule is called a {\it dual
$\mathfrak{A}$-bimodule}.

 A {\it derivation} $D : \mathfrak{A}\rightarrow X$ is a continuous linear map such that $D(ab) =
a \cdot D(b) + D(a)\cdot b$, for all $a, b \in \mathfrak{A}$. Given $x\in X$, the {\it inner derivation}
$\delta_x : \mathfrak{A}\rightarrow X$, is defined by $\delta_x (a) = a\cdot x- x \cdot a$.

 According to Johnson¡¯s
original definition, a Banach algebra $\mathfrak{A}$ is {\it amenable} if every
 derivation from $\mathfrak{A}$ into the dual  $\mathfrak{A}$-bimodule
$X^*$ is inner for all Banach $\mathfrak{A}$-bimodules $X$. If $T\in \mathfrak{B}(\mathfrak{H})$, denote the
norm-closure of span$\{T^k: k\in \{0\}\cup \mathbb{N}\}$ by $\mathfrak{A}_T$, where $\mathbb{N}$ is the set of
natural numbers, $T$ is said to be an {\it amenable operator}, if $\mathfrak{A}_T$ is an amenable Banach
algebra. Ever since its introduction, the concept of amenability has played an important role in research in
Banach algebras, operator algebras and harmonic analysis. There has been a long-standing conjecture in the
Banach algebra community, stated as follows:

\begin{conj}\label{conj 1}
A commutative Banach subalgebra of $\mathfrak{B}(\mathfrak{H})$ is amenable if and only if it is similar to a
$C^*$-algebra.
\end{conj}

 One of the
first result in this direction is due to Willis \cite{W1995}. Willis showed that if $T$ is an amenable compact
operator, then $T$ is similar to a normal operator. In \cite{G2006} Gifford studied the reduction property for
operator algebras consisting of compact operators and showed that if such an algebra is amenable then it is
similar to a $C^*$-algebras. In the recent papers \cite{F2005}, \cite{F2007} Farenick, Forrest and Marcoux
showed that  if $T$ is similar to a normal operator, then $\mathfrak{A}_T$ is amenable if and only if
$\mathfrak{A}_T$ is similar to a $C^*$-algebra and the spectrum of $T$ has connected complement and empty
interior; If $T$ is a triangular operator with respect to an orthonormal basis of $\mathfrak{H}$, then
$\mathfrak{A}_T$ is amenable if and only if $T$ is similar to a normal operator  whose spectrum has connected
complement and empty interior. For further details see \cite{F2005} and \cite{F2007}.

In this paper, we give the characterization of the structure of amenable operators. At first, we use the
reduction theory of von Neumann to give two equivalent descriptions for Conjecture 1.1; and then, we give two
decompositions for amenable operators, which supporting the Conjecture 1.1.

\vskip1cm
\section{An equivalent formulation of the conjecture 1.1}
In this section we use the reduction theory of von Neumann to give two equivalent descriptions for Conjecture
1.1. We obtain that  every amenable operator is similar to a normal operator if and only if every non-scalar
amenable operator has a non-trivial hyperinvariant subspace if and only if every amenable operator is similar to
a reducible operator and has a non-trivial invariant subspace.

In order to proof the main theorem, we need to introduce von Neumann's reduction theory \cite{S1967} and some
lemmas.

Let
$\mathfrak{H_1}\subseteq\mathfrak{H_2}\subseteq\cdots\subseteq\mathfrak{H_\infty}$
be a sequence of Hilbert spaces chosen once and for all,
$\mathfrak{H_n}$ having the dimension $n$. Let $\mu$ be a finite
positive regular measure defined on the Borel sets of a separable
metric space $\wedge$, and let $\{E_n\}_{n=1}^\infty$ be a
collection of disjoint Borel sets of $\wedge$ with union $\wedge$.
Then the symbol
$$\int_\wedge^\oplus \mathfrak{H}(\lambda)\mu(d\lambda)$$
denotes the set of all functions $f$ defined on $\wedge$ such that

(1)$f(\lambda)\in \mathfrak{H}_n\subseteq\mathfrak{H}_\infty$ if
$\lambda\in E_n$;

(2)$f(\lambda)$ is a $\mu$-measurable function with values in
$\mathfrak{H}_\infty$;

(3) $\int_\wedge^\oplus |f(\lambda)|^2\mu(d\lambda)<\infty$.

We put

(4)$(f, g)=\int_\wedge^\oplus (f(\lambda), g(\lambda))\mu(d\lambda).$ \\
The set of functions thus defined is
called the {\it direct integral Hilbert space with measure $\mu$ and dimension sets $\{E_n\}$} and denoted by
$\mathfrak{H}=\int_\wedge^\oplus \mathfrak{H}(\lambda)\mu(d\lambda)$.

An operator on $\mathfrak{H}$ is said to be {\it decomposable} if
there exists a strongly $\mu$-measurable operator-value function
$A(\cdot)$ defined on $\wedge$ such that
 $A(\lambda)$ is a
bounded operator on the space $\mathfrak{H}(\lambda)=\mathfrak{H}_n$ when $\lambda\in E_n$, and for all
$f\in\mathfrak{H}$, $(Af)(\lambda)=A(\lambda)f(\lambda)$. We write $A=\int_\wedge^\oplus
A(\lambda)\mu(d\lambda)$ for the equivalence class corresponding to $A(\cdot)$. If $A(\lambda)$ is a scalar
multiple of the identity on $\mathfrak{H}(\lambda)$ for almost all $\lambda$, then $A$ is called {\it diagonal}.
The collection of all diagonal operator is called the {\it diagonal algebra} of $\wedge$. In \cite{S1967}I.3,
Schwartz showed that an operator $A$ on Hilbert space $\mathfrak{H}=\int_\wedge^\oplus
\mathfrak{H}(\lambda)\mu(d\lambda)$ is decomposable if and only if $A$ belong to the commutant of the diagonal
algebra of $\wedge$. And $||A||=\mu-ess.sup_{\lambda\in\wedge}||A(\lambda)||$.

In \cite{A 1977}, Azoff, Fong and Gilfeather used von Neumann's reduction theory to define the reduction theory
for non-selfadjoint operator algebras: Fix a partitioned measure space $\wedge$ and let $\mathfrak{D}$ be the
corresponding diagonal algebra. Given an algebra $\mathfrak{A}$ of decomposable operators. Each operator
$A\in\mathfrak{A}$ has a decomposition $A=\int_\wedge^\oplus A(\lambda)\mu(d\lambda)$. Chosse a countable
generating set $\{A_n\}$ for $\mathfrak{A}$. let $\mathfrak{A}(\lambda)$ be the strongly closed algebra
generated by the $\{A_n(\lambda)\}$. $\mathfrak{A}\sim\int_\wedge^\oplus \mathfrak{A}(\lambda)\mu(d\lambda)$ is
called the {\it decomposition of} $\mathfrak{A}$ {\it respect to} $\mathfrak{D}$. A decomposition
$\mathfrak{A}\sim\int_\wedge^\oplus \mathfrak{A}(\lambda)\mu(d\lambda)$ of an algebra is said to be {\it
maximal} if the corresponding diagonal algebra is maximal among the abelian von Neumman subalgebras of
$\mathfrak{A}'$. The following lemma is a basic result in \cite{A 1977} which will be used in this paper.

\begin{lem}\label{Fong}(\cite{A 1977}, Theorem 4.1)
Let $\mathfrak{A}\sim\int_\wedge^\oplus
\mathfrak{A}(\lambda)\mu(d\lambda)$ be a decomposition of a
reductive algebra. Then almost all of $\{\mathfrak{A}(\lambda)\}$
are reducible. In particular, if the decomposition is maximal, then
almost all of the algebras $\{\mathfrak{A}(\lambda)\}$ are
transitive.
\end{lem}

In \cite{G2006} Gifford studied the reduction property for operator
algebras and obtained the following result:

\begin{lem}\label{lem 1}(\cite{G2006} Lemma 4.4, Lemma 4.12)
If $\mathfrak{A}$ is a commutative  amenable operator algebra, then
$\mathfrak{A}',\mathfrak{A}''$ are complete reducible and there
exists $M\geq 1$ so that for any idempotent $p\in\mathfrak{A}''$
$||p||\leq M$.
\end{lem}

Assume $\mathfrak{A}$ is a operator algebra, let $P(\mathfrak{A})$ denote the idempotents in $\mathfrak{A}$ and
$\mathcal{P}(\mathfrak{A})$ denote the strongly closed algebra generated by $P(\mathfrak{A})$. We get the
following lemma:

\begin{lem}\label{lem 2}
If $\mathfrak{A}$ is a commutative amenable operator algebra, then $\mathcal{P}(\mathfrak{A''})$ is similar to
an abelian von Neumann algebra.
\end{lem}

\begin{proof}
By Lemma \ref{lem 1} and \cite[ Corollary 17.3]{D1988}, it follows that there exists $X\in
\mathfrak{B}(\mathfrak{H})$ such that $XpX^{-1}$ is selfadjoint for each $p\in P(\mathfrak{A''})$. Hence
$\mathcal{P}(\mathfrak{A''})$ is similar to a abelian von Neumann algebra.
\end{proof}

\begin{lem}\label{lem 3}(\cite{F2005})
Let $\mathfrak{A}$ and $\mathfrak{B}$ be Banach algebras and suppose
that $\varphi: \mathfrak{A}\longrightarrow\mathfrak{B}$ is a
continuous homomorphism with $\varphi(\mathfrak{A})$ dense in
$\mathfrak{B}$. If $\mathfrak{A}$ is amenable, then $\mathfrak{B}$
is amenable.
\end{lem}

\begin{nota}\label{nota}
From Lemma \ref{lem 2},\ref{lem 3} we always assume that $\mathcal{P}(\mathfrak{A}_T'')$ is a abelian von
Neumann algebra, and $\mathfrak{A}_T'$ is a reducible operator algebra in this section.
\end{nota}

Now we will proof the main result of this section:

\begin{thm}\label{thm1}
The following are equivalent:

\textup{(1)} Every amenable operator is similar to a normal operator;

\textup{(2)} Every non-scalar amenable operator has a non-trivial hyperinvariant subspace;

\textup{(3)} Every amenable Banach algebra which is generated by an operator is similar to a $C^*$-algebra.
\end{thm}

\begin{proof}

$(1)\Leftrightarrow(3)$ and $(1)\Rightarrow(2)$ is clear by
\cite{F2005}. Therefore, in order to establish the theorem it
suffices to show the implications $(2)\Rightarrow(1)$.

Assume (2), by Lemma \ref{lem 2} choose a maximal decomposition for $\mathfrak{A}_T'\sim \int_\wedge^\oplus
\mathfrak{A}_T'(\lambda)\mu(d\lambda)$ respect to the diagonal algebra $\mathcal{P}(\mathfrak{A}_T'')$.

Assume $T\sim \int_\wedge^\oplus T(\lambda)\mu(d\lambda)$ is the
decomposition for $T$. Let $\{p_n\}_{n=1}^\infty$ denote the all
rational polynomials. Then $p_n(T)\sim \int_\wedge^\oplus
p_n(T)(\lambda)\mu(d\lambda)$ is decomposable for all $n$ and there
exists a measurable
 $E\subseteq\wedge$ such that $\mu(\wedge-E)=0$ and for any $\lambda\in
 E$ we have
 $ p_n(T)(\lambda)= p_n(T(\lambda))$ and
$||p_n(T)(\lambda)||\leq ||p_n(T)||$ by \cite[Lemma I.3.1, I.3.2]{S1967}. Define a mapping $\varphi_\lambda :
\mathfrak{A}_{T}\rightarrow \mathfrak{A}_{T(\lambda)}$ by $\varphi_\lambda(p_n(T))=p_n(T(\lambda))$ for each
rational polynomial $p_n$ and $\lambda\in E$. Note that $||p_n(T(\lambda))||\leq ||p_n(T)||$ for each  rational
polynomial $p_n$ and furthermore,  $\{p_n(T)\}$ is dense in $\mathfrak{A}_{T}$. Hence, $\varphi_\lambda$ is
well-defined and $\varphi_\lambda$ a continuous homomorphism with $\varphi(\mathfrak{A}_T)$ dense in
$\mathfrak{A}_{T(\lambda)}$. By Lemma \ref{lem 3}, $T(\lambda)$ is amenable for almost all $\lambda$.

Now for almost all $\lambda$, $T(\lambda)$ is amenable and $\mathfrak{A}_T'(\lambda)\subseteq
\mathfrak{A}_{T(\lambda)}'$ and $\mathfrak{A}_T'(\lambda)$ is transitive by Lemma \ref{Fong}. Thus almost all of
$T(\lambda)$ are scalar operators, i.e. $T$ is a normal operator.
\end{proof}

\begin{cor}
Every amenable operator is similar to a normal operator if and only
if there exists a non-trivial idempotent in the double-commutant of
every non-scalar amenable operator.
\end{cor}

\begin{rem}\label{rem1}
In \cite{F2005}, Farenick, Forrest and Marcoux showed that if $T\in\mathfrak{B}(\mathfrak{H})$ is amenable and
similar to a
 normal operator $N$, then the spectrum of $N$ has connected complement and empty
interior. According to \cite[Theorem 1.23]{R2003}, $N$ is a reducible operator. Hence, there exists an
invertible operator $X\in\mathfrak{B}(\mathfrak{H})$ such that $\mathfrak{A}_{XTX^{-1}}''$ is a reducible
algebra. The following theorem give the equivalent description for Conjecture 1.1 from the existence of
invariant subspace for amenable operators.
\end{rem}

\begin{thm}\label{thm11}
The following are equivalent:

\textup{(1)} Every amenable operator is similar to a normal operator;

\textup{(2)} For every amenable operator $T\in\mathfrak{B}(\mathfrak{H})$, there exists an invertible operator
$X\in\mathfrak{B}(\mathfrak{H})$ such that $\mathfrak{A}_{XTX^{-1}}''$ is a reducible algebra and $T$ has a
non-trivial invariant subspace.
\end{thm}

\begin{proof}
$(1)\Rightarrow(2)$ is clear by Remark \ref{rem1}.

 $(2)\Rightarrow(1)$ is trivial modifications adapt the proof of theorem 2.6.
\end{proof}

\begin{rem}
According to theorem 2.6, \ref{thm11}, we obtain that the Conjecture \ref{conj 1} for  operator algebra which is
generated by an operator is equivalent to the following statements:

(1) Every amenable operator $T$ has a non-trivial invariant subspace and renorm $\mathfrak{H}$ with an
equivalent Hilbert space norm so that under this norm Lat$\mathfrak{A}_T$ becomes orthogonally complemented;

(2) Every non-scalar amenable operator has a non-trivial hyperinvariant subspace.
\end{rem}
\vskip1cm
\section{Decomposition of  Amenable operators}

In this section, we get two decompositions for amenable operators and prove that the two decompositions are the
same which supporting Conjecture 1.1.

At first, we summarize some of the details of  multiplicity theory for abelian von Neumann algebras. For the
most part, we will follow \cite{D1996}. If $A$ is an operator on a Hilbert space $\mathfrak{K}$ and $n$ is a
cardinal number, Let $\mathfrak{K}^n$ denote the orthogonal direct sum of $n$ copies of $\mathfrak{K}$, and
$A^{(n)}$ be the operator on $\mathfrak{K}^n$ which is the direct sum of $n$ copies of $A$. Whenever
$\mathfrak{A}$ is an operator algebra on $\mathfrak{K}$, $\mathfrak{A}^{(n)}$ denotes the algebra $\{A^{(n)},
A\in\mathfrak{A}\}$. An abelian von Neumann algebra $\mathfrak{B}$ is of {\it uniform multiplicity $n$} if it is
(unitary equivalent to) $\mathfrak{A}^{(n)}$ for some maximal abelian von Neumann algebra $\mathfrak{A}$. By
\cite{D1996}, for any abelian von Neumann algebra $\mathfrak{A}$, there exists a sequens of regular Borel
measures $\{\mu_n\}$ on a sequens of separable metric space $\{X_n\}$ such that $\mathfrak{A}$ is unitary
equivalent to $\sum_{n=1}^\infty\oplus \mathfrak{B}_n\oplus\mathfrak{B}_\infty$, where $\mathfrak{B}_n$ is a von
Neumman algebra which has uniform multiplicity $n$ for all $1\leq n\leq \infty$. For further details see
\cite[II.3]{D1996}.

\begin{prop}\label{prop1}
Suppose that $T$ is amenable operator and $\mathfrak{A}_T'$ contains a subalgebra which is similar to an abelian
von Neumman algebra with no direct summand of uniform multiplicity infinite, then $T$ is similar to a normal
operator.
\end{prop}
\begin{proof}
For the sake of simplicity, we assume $\mathfrak{A}_T'$ contains a subalgebra $\mathfrak{B}$ which is an abelian
von Neumman subalgebra with no direct summand of  uniform multiplicity infinite. Trivial modifications adapt the
proof to the more general case.

By \cite[II.3]{D1996}, there exists a sequens of regular Borel measures $\{\mu_n\}$ on a sequens of separable
metric space $\{X_n\}$ such that $\mathfrak{B}$ is unitarily equivalent to $\sum_{n=1}^\infty\oplus
\mathfrak{B}_n$, where $\mathfrak{B}_n$ is a von Neumman algebra which has
 uniform multiplicity $n$ for all $n$. Hence, $T=\sum_{n=1}^\infty\oplus T_n$,
where $T_n\in \mathfrak{B}_n'$. It suffices to show that $T_n$ is similar to a normal operator for all $n$, then
by \cite[Corollary 26]{F1977}, it follows that $T$ is similar to a normal operator.

Since $T\in\mathfrak{B}_n'$, according to \cite[Theorem 7.20]{R2003}, for any $1\leq n<\infty$ there exists a
unitary operator $U_n\in\mathfrak{B}_n'$ such that
$$U_nT_n(U_n)^{-1}= \left[\begin{array}{ccccc}
N_{11}&N_{12}&\cdots&\cdots&N_{1n}\\
0&N_{22}&\cdots&\cdots&N_{2n}\\
0&0&\ddots&\ddots&\vdots\\
\vdots&\vdots&\ddots&\ddots&\vdots\\
0&0&\cdots&0&N_{nn}\\
\end{array}\right]
 $$
 where $N_{ij}$ is a normal operator for all $1\leq i,j\leq n$. By \cite[ Proposition 3.1]{Y}, it follows that $T_n$ is similar to $\oplus_{i=1}^{n}
 N_{ii}$, i.e. $T_n$ is similar to a normal operator for all $n$.
\end{proof}

\begin{cor}\label{cor1}
Assume $T$ is amenable operator, then there exists hyperinvariant subspaces $M, N$ of $T$ such that $T$ has the
form $T=T_1\dot{+}T_2$ respect to the space decomposition $\mathfrak{H}=M\dot{+}N$, where $T_1, T_2$ are
amenable operators, $T_1$ is similar to a normal operator and $\mathcal{P}(\mathfrak{A}_{T_2}'')$ is similar to
an abelian von Neumman algebra with uniform multiplicity infinite.
\end{cor}

The proof of  the following lemma is straightforward and we  omit
it.

\begin{lem}\label{lem4}
Suppose that $\mathfrak{A}$ is a completely reductive operator algebra and $p\in P(\mathfrak{A}')$. Then
$p\mathfrak{A}$ is a completely reductive operator algebra on $\textup{Ran} p$.
\end{lem}

We are in need of the following propositions before we can address the main theorem of this section.

\begin{prop}\label{prop6}
Assume that $T$ is a amenable operator and  there exists a space decomposition $\mathfrak{H}=M\dot{+}N$ such
that $T$ has the matrix form $T= \left[\begin{array}{cc}
T_{1}&\\
&T_{2}
\end{array}\right
]
\begin{matrix}
\mbox{$M$}\\
\mbox{$N$}\\
\end{matrix}
$. Then $T$ is similar to a normal operator if and only if $T_1$ and $T_2$ are similar to normal operators.
\end{prop}

\begin{proof}
Assume that  $T$ has the matrix form $$T= \left[\begin{array}{cc}
T_{1}&T_{12}\\
&\widetilde{T_{2}}
\end{array}\right
]
\begin{matrix}
\mbox{$M$}\\
\mbox{$M^{\bot}$}\\
\end{matrix}$$
respect to the space decomposition $\mathfrak{H}=M\oplus M^{\bot}$. By \cite[Lemma 2.8]{Y}, there exists an
invertible operator $S=\left[\begin{array}{cc}
I&S_{12}\\
&I
\end{array}\right
]
\begin{matrix}
\mbox{$M$}\\
\mbox{$M^{\bot}$}\\
\end{matrix}$ such that $S^{-1}TS=\left[\begin{array}{cc}
T_{1}&\\
&\widetilde{T_{2}}
\end{array}\right
]
\begin{matrix}
\mbox{$M$}\\
\mbox{$M^{\bot}$}\\
\end{matrix}$. Assume that  $S$ has the matrix form
$S=\begin{array}{cc}
\begin{array}{cc}M&M^{\bot}
\end{array}\\
\left[\begin{array}{cc}
I&\\
&S_1
\end{array}\right
]\end{array}
\begin{matrix}
\\
\mbox{$M$}\\
\mbox{$N$}\\
\end{matrix}
$, we obtain that $T_2=S_1\widetilde{T_{2}}S_1^{-1}$. By \cite[propsition 6.5]{H1978}, we get that $T$ is
similar to a normal operator if and only if $T_1$ and $T_2$ are similar to normal operators.
\end{proof}

\begin{prop}\label{prop5}
Suppose that $T$ is an amenable operator, $M_1\in$
Lat$\mathfrak{A}_T'$ and $M_2\in$ Lat$\mathfrak{A}_T''$. Then
$M_1+M_2$ is closed.

Moreover, if $T|_{M_1}$ and $T|_{M_2}$ are similar to normal operators, then $T|_{M_1+M_2}$ is similar to a
normal operator.
\end{prop}
\begin{proof}
Let $N_0=M_1\cap M_2$, according to  Lemma \ref{lem4}, there exists $N\in$ Lat$\mathfrak{A}_{T}''$ such that
$M_2=N_0\dot{+}N$. Choose $q\in P(\mathfrak{A}_{T}')$,  such that $\textup{Ran} q=N$. By the assumption,
$M_1\in$ Lat$\mathfrak{A}_T'$. Hence $qM_1\subset M_1\cap N=\{0\}$. Therefore $M_1\subset (I-q)\mathfrak{H}$. We
see that $M_1+M_2=M_1\dot{+}N$ is closed. This establishes the first statement of the proposition.

Since $T|_{M_2}$ is similar to a normal operator, by proposition \ref{prop6}, we get that $T|_N$ is similar to a
normal operator. By the assumption $T|_{M_1}$ is similar to a normal operator, using proposition \ref{prop6}
again, we obtain that $T|_{M_1+M_2}=T|_{M_1\dot{+}N}$ is similar to a normal operator.
\end{proof}

Now we will obtain the main theorem of this section.

\begin{thm}\label{thm2}
Assume $T$ is an amenable operator, then there exists hyperinvariant subspaces $M_1, M_2$ of $T$ such that $T$
has the form $T=T_1\dot{+}T_2$ respect to the space decomposition $\mathfrak{H}=M_1\dot{+}M_2$
 and satisfies that:

\textup{(1)} $T_1,T_2$ are amenable operators;

\textup{(2)} If $M$ is a hyperinvariant subspace of $T$ and $T|_M$ is similar to a normal operator, then
$M\subseteq M_1$, \textup{i.e.} $M_1$ is the largest hyperinvariant subspace on which $T$ is similar to a normal
operator;

\textup{(3)} For any $q\in P(\mathfrak{A}_{T_2}'')$, $T_2|_{\textup{Ran} q}$ is not similar to a normal
operator;

\textup{(4)} $\mathcal{P}(\mathfrak{A}_{T_2}'')$ is similar to an abelian von Neumman algebra with uniform
multiplicity infinite;

\textup{(5)} $\mathfrak{A}_{T}'=\mathfrak{A}_{T_1}'\dot{+}\mathfrak{A}_{T_2}'$,
$\mathfrak{A}_{T}''=\mathfrak{A}_{T_1}''\dot{+}\mathfrak{A}_{T_2}''$;

\textup{(6)} There exists no nonzero compact operator in $\mathfrak{A}_{T_2}'$.
\end{thm}
\begin{proof}
Case1. For any $p\in P(\mathfrak{A}_{T}'')$, $T|_{\textup{Ran} p}$ is not similar to  normal operator. According
to the proof of Proposition \ref{prop1}, we obtain that $\mathcal{P}(\mathfrak{A}_{T}'')$ is similar to an
abelian von Neumman algebra with uniform multiplicity infinite. Let $M_1=0$.

Case2. There exists $p\in P(\mathfrak{A}_{T}'')$ such that $T|_{\textup{Ran} p}$ is similar to normal operator.
Then, by Zorn's Lemma and the same method in the proof of \cite[Corollary 26]{F1977}, we can show that there
exists an element $p_0\in P(\mathfrak{A}_{T}'')$ which is maximal with respect to the property that
$T|_{\textup{Ran} p_0}$ is similar to a normal operator. Using Proposition \ref{prop5}, $\textup{Ran} p_0$ is
the largest hyperinvariant subspace of $T$ on which $T$ is similar to a normal operator. Hence, $T$ has the form
$T=T_1\dot{+} T_2$ with respect to the space decomposition $\mathfrak{H}=\textup{Ran} p_0\dot{+} \textup{Ker}
p_0$ where $T_1$ is similar to a normal operator, $T_1,T_2$ are amenable operators. Let $M_1=\textup{Ran} p_0,
M_2=\textup{Ker} p_0$.

Next we will prove that for any $q\in P(\mathfrak{A}_{T_2}'')$, $T_2|_{\textup{Ran} q}$ is not similar to
normal operator. Then according to Proposition \ref{prop1} $\mathcal{P}(\mathfrak{A}_{T_2}'')$ is similar to an
abelian von Neumman algebra with uniform multiplicity infinite.

Indeed, if there exists $q\in P(\mathfrak{A}_{T_2}'')$ such that $T_2|_{\textup{Ran} q}$ is similar to a normal
operator and $q$ has the form $q= \left[\begin{array}{cc}
I&0\\
0&0\\
\end{array}\right]
 \begin{matrix}
\mbox{$\textup{Ran} q$}\\
\mbox{$\textup{Ker} q$}\\
\end{matrix}$. Then for any $A\in \mathfrak{A}_{T}'$, $A$ has the form
$$A= \left[\begin{array}{ccc}
A_{11}&&\\
&A_{22}&\\
&&A_{33}\\
\end{array}\right]
 \begin{matrix}
\mbox{$\textup{Ran} p_0$}\\
\mbox{$\textup{Ran} q$}\\
\mbox{$\textup{Ker} q$}\\
\end{matrix}.$$
 Let
$$R= \left[\begin{array}{ccc}
I&&\\
&I&\\
&&0\\
\end{array}\right]
 \begin{matrix}
\mbox{$\textup{Ran} p_0$}\\
\mbox{$\textup{Ran} q$}\\
\mbox{$\textup{Ker} q$}\\
\end{matrix}.$$
Then $R\in P(\mathfrak{A}_{T}'')$. By the assumption $T|_{\textup{Ran} R}$ is similar to a normal operator which
contradicts to the maximal property of $p_0$.

At last we will prove that there exists no nonzero compact operator in $\mathfrak{A}_{T_2}'$.

Indeed, if there exists a nonzero compact operator $k_0\in\mathfrak{A}_{T_2}'$, let $L_1$ denote the subspace
spanned by the ranges of all compact operators in $\mathfrak{A}_{T_2}'$, and $L_2$ the intersection of their
kernel, by \cite[Lemma 3.1]{R1993}, both $L_1, L_2$ lie in Lat$\mathfrak{A}_{T_2}'$ and
$L_1\dot{+}L_2=\textup{\textup{Ker}} p_0$. Considering the restricting $T_2|_{L_1}$, assume $T_{21}=T_2|_{L_1}$,
then $T_{21}$ is an amenable operator and $\mathfrak{A}_{T_{21}}'$ contain a sufficient set of compact
operators. By Lemma \ref{lem 1} and \cite[Theorem 9]{R1982}, $T_{21}$ is similar to a normal operator which
contradicts to the above discussion.
\end{proof}

Trivial modifications adapt the proof of Theorem \ref{thm2}, we obtain the following theorem which decomposes
amenable operators by the invariant subspaces of them. The proof is similar to Theorem \ref{thm2} and we omit
it.

\begin{thm}\label{thm22}
Assume $T$ is an amenable operator, then there exists invariant subspaces $N_1, N_2$ of $T$ such that $T$ has
the form $T=A_1\dot{+}A_2$ respect to the space decomposition $\mathfrak{H}=N_1\dot{+}N_2$
 and satisfies that:

\textup{(1)} $A_1,A_2$ are amenable operators;

\textup{(2)} If $N$ is an invariant subspace of $T$ such that $N_1\subseteq N$ and $T|_N$ is similar to a normal
operator, then $N= N_1$, \textup{i.e.} $N_1$ is the maximal invariant subspace on which $T$ is similar to a
normal operator;

\textup{(3)} For any $q\in P(\mathfrak{A}_{T_2}')$, $T_2|_{\textup{Ran} q}$ is not similar to a normal operator;

\textup{(4)} If  $\mathcal{P}(\mathfrak{A}_{T_2}')$ contains a subalgebra which is similar to an abelian von
Neumman algebra then the von Neumman algebra has the uniform multiplicity infinite.
\end{thm}

\begin{rem}
If the answer to Conjecture 1.1 is positive, by Theorem 2.6, every amenable is similar to a normal operator.
Then, for the above theorem $M_1=N_1=\mathfrak{H}$. That is to say, the two decompositions of theorem
 3.6 and 3.7 are the same. The remainder of this section, we will prove that
the two decompositions are the same which supporting Conjecture 1.1.
\end{rem}

\begin{lem}\cite{F1977}\label{lem5}
If $T\in \mathfrak{B}(\mathfrak{H})$ is an amenable operator and there exist a one-to-one bounded linear map
$W:\mathfrak{H}\rightarrow\mathfrak{H}_2$, a bounded linear map $V:\mathfrak{H}_1\rightarrow\mathfrak{H}$ with
dense range and  operators $S_1\in \mathfrak{B}(\mathfrak{H}_1)$, $S_2\in \mathfrak{B}(\mathfrak{H}_2)$ which
are similar to normal operators such that $TV=VS_1$ and $WT=S_2W$, then $T$ is similar to a normal operator.
\end{lem}

\begin{cor}
Assume $T=B_1B_2$ is an amenable operator, where $B_1,B_2$ are positive operators, then $T$ is similar to a
normal operator.
\end{cor}

\begin{proof}
Assume $B_1, B_2$ have the forms
$$B_2=\left[\begin{array}{cc}
0&\\
&\widetilde{B_2}\\
\end{array}\right],
B_1=\left[\begin{array}{cc}
B_{11}&B_{12}\\
B_{12}^*&B_{22}\\
\end{array}\right],$$ respect to the space decomposition $\mathfrak{H}=\textup{Ker} B_2\oplus(\textup{Ker} B_2)^\bot$  where $\widetilde{B_2}$ is one-to-one and $B_{11}, B_{22}$ are
positive operators. Thus $T$
 has the form $T=\left[\begin{array}{cc}
0&B_{12}\widetilde{B_2}\\
0&B_{22}\widetilde{B_2}\\
\end{array}\right]$
respect to the space decomposition. Since $T$ is an amenable operator, by \cite[Lemma 2.8]{Y}, $T$ is similar to
$\left[\begin{array}{cc}
0&0\\
0&B_{22}\widetilde{B_2}\\
\end{array}\right]$. Thus without loss of generality, we may assume that $B_2$ is one-to-one.

Assume that $B_1, B_2$ has the form
$$B_1=\left[\begin{array}{cc}
\widetilde{B_1}&\\
&0\\
\end{array}\right],
B_2=\left[\begin{array}{cc}
B_{11}&B_{12}\\
B_{12}^*&B_{22}\\
\end{array}\right],$$
respect to the space decomposition $\mathfrak{H}=(\textup{Ker} B_1)^\bot\oplus \textup{Ker} B_1$ where
$\widetilde{B_1}$ is one-to-one and has dense range and $B_{11}, B_{22}$ are positive operators. Thus $T$
 has the form $T=\left[\begin{array}{cc}
\widetilde{B_1}B_{11}&\widetilde{B_1}B_{12}\\
0&0\\
\end{array}\right]$
respect to the space decomposition. Since $T$ is an amenable operator, by \cite[Lemma 2.8]{Y}, $T$ is similar to
$\left[\begin{array}{cc}
\widetilde{B_1}B_{11}&0\\
0&0\\
\end{array}\right]$ and there exists an operator $S$ such that $\widetilde{B_1}B_{12}=\widetilde{B_1}B_{11}S$.
Note that $\widetilde{B_1}, B_2$ are one-to-one,  hence $B_{12}=B_{11}S$, and $B_{11}$ is one-to-one. Thus
without loss of generality, we may assume that $B_1$ has dense range and $B_2$ is one-to-one.

Note that $B_1^{\frac{1}{2}}B_2B_1^{\frac{1}{2}}, B_2^{\frac{1}{2}}B_1B_2^{\frac{1}{2}}$ are positive operators
and $TB_1^{\frac{1}{2}}=B_1^{\frac{1}{2}}B_1^{\frac{1}{2}}B_2B_1^{\frac{1}{2}}$ and
$B_2^{\frac{1}{2}}T=B_2^{\frac{1}{2}}B_1B_2^{\frac{1}{2}}B_2^{\frac{1}{2}}$, by Lemma \ref{lem5}, $T$ is similar
to a normal operator.
\end{proof}

\begin{thm}
The two decompositions for an amenable operator in  Theorem \ref{thm2}, \ref{thm22} are the same.
\end{thm}

\begin{proof}
According to Theorem \ref{thm2}, \ref{thm22}, and Proposition \ref{prop5}, it is suffices to proof that
$N_1\in
Lat \mathfrak{A}_T'$.

In fact, if not. $T$ has the form $T= \left[\begin{array}{cc}
T_1&\\
&T_2\\
\end{array}\right]
 \begin{matrix}
\mbox{$N_1$}\\
\mbox{$N_2$}\\
\end{matrix}$
and there exists
 $S= \left[\begin{array}{cc}
0&0\\
Y&0\\
\end{array}\right]
 \begin{matrix}
\mbox{$N_1$}\\
\mbox{$N_2$}\\
\end{matrix}\in \mathfrak{A}_T'$ where $Y\neq 0$.
Note that $S$ and $T$ have the form $$S= \left[\begin{array}{ccc}
0&0&0\\
\tilde{Y}&0&0\\
0&0&0\\
\end{array}\right]
 \begin{matrix}
\mbox{$N_1$}\\
\mbox{$\overline{\textup{Ran} Y}$}\\
\mbox{$N_2\ominus\overline{\textup{Ran} Y}$}\\
\end{matrix},
T= \left[\begin{array}{ccc}
T_1&0&0\\
&T_{21}&T_{22}\\
&T_{23}&T_{24}\\
\end{array}\right]
 \begin{matrix}
\mbox{$N_1$}\\
\mbox{$\overline{\textup{Ran} Y}$}\\
\mbox{$N_2\ominus\overline{\textup{Ran} Y}$}\\
\end{matrix},$$
where $\tilde{Y}$ has dense range. Note that $TS=ST$, we get that $T_{23}=0$. Since $T$ is amenable, by
\cite[Lemma 2.8]{Y} there exists an operator $B: N_2\ominus\overline{\textup{Ran} Y}\rightarrow
\overline{\textup{Ran} Y}$ such that
$$\left[\begin{array}{ccc}
I&0&0\\
&I&B\\
&&I\\
\end{array}\right]
\left[\begin{array}{ccc}
T_1&0&0\\
&T_{21}&T_{22}\\
&&T_{24}\\
\end{array}\right]
\left[\begin{array}{ccc}
I&0&0\\
&I&-B\\
&&I\\
\end{array}\right]
=\left[\begin{array}{ccc}
T_1&0&0\\
&T_{21}&0\\
&&T_{24}\\
\end{array}\right].$$ Moreover,

$$\left[\begin{array}{ccc}
I&0&0\\
&I&B\\
&&I\\
\end{array}\right]
\left[\begin{array}{ccc}
0&0&0\\
\tilde{Y}&0&0\\
0&0&0\\
\end{array}\right]
 \left[\begin{array}{ccc}
I&0&0\\
&I&-B\\
&&I\\
\end{array}\right]=
\left[\begin{array}{ccc}
0&0&0\\
\tilde{Y}&0&0\\
0&0&0\\
\end{array}\right].$$
Hence, we can assume that $Y$ has dense range. Using $T$ is amenable again, there exists
$L=\left[\begin{array}{cc}
0&X\\
0&0\\
\end{array}\right]
 \begin{matrix}
\mbox{$N_1$}\\
\mbox{$N_2$}\\
\end{matrix}\in \mathfrak{A}_T'$, where $X\neq 0$, by \cite[ lemma 4.11]{G2006}.
Similar to the decomposition to $S$ and $T$, we get that $S$, $L$ and $T$ have the form
$$S=
\left[\begin{array}{ccc}
0&0&0\\
\tilde{Y_1}&0&0\\
\tilde{Y_2}&0&0\\
\end{array}\right],
L= \left[\begin{array}{ccc}
0&0&\tilde{X}\\
0&0&0\\
0&0&0\\
\end{array}\right],
T= \left[\begin{array}{ccc}
T_1&0&0\\
&T_{31}&T_{32}\\
&T_{33}&T_{34}\\
\end{array}\right],$$
respect to the space decomposition $\mathfrak{H}=N_1\oplus \textup{Ker} X\oplus(N_2\ominus \textup{Ker} X)$,
where $\tilde{X}$ is one-to-one, and $\tilde{Y_1},\tilde{Y_2}$ has dense range. Note that $LT=TL$, we get that
$T_{33}=0$.
 Using $T$ is amenable again, there exists an operator
$C: N_2\ominus \textup{Ker} X\rightarrow \textup{Ker} X$ such that
$$\left[\begin{array}{ccc}
I&0&0\\
&I&C\\
&&I\\
\end{array}\right]
\left[\begin{array}{ccc}
T_1&0&0\\
&T_{31}&T_{32}\\
&&T_{34}\\
\end{array}\right]
\left[\begin{array}{ccc}
I&0&0\\
&I&-C\\
&&I\\
\end{array}\right]
=\left[\begin{array}{ccc}
T_1&0&0\\
&T_{31}&0\\
&&T_{34}\\
\end{array}\right]$$

$$\left[\begin{array}{ccc}
I&0&0\\
&I&C\\
&&I\\
\end{array}\right]
\left[\begin{array}{ccc}
0&0&\tilde{X}\\
0&0&0\\
0&0&0\\
\end{array}\right]
\left[\begin{array}{ccc}
I&0&0\\
&I&-C\\
&&I\\
\end{array}\right]
=\left[\begin{array}{ccc}
0&0&\tilde{X}\\
0&0&0\\
0&0&0\\
\end{array}\right]$$

and

$$\left[\begin{array}{ccc}
I&0&0\\
&I&C\\
&&I\\
\end{array}\right]
\left[\begin{array}{ccc}
0&0&0\\
\tilde{Y_1}&0&0\\
\tilde{Y_2}&0&0\\
\end{array}\right]
\left[\begin{array}{ccc}
I&0&0\\
&I&-C\\
&&I\\
\end{array}\right]
=\left[\begin{array}{ccc}
0&0&0\\
\tilde{Y_1}+C\tilde{Y_2}&0&0\\
\tilde{Y_2}&0&0\\
\end{array}\right].$$
Moreover, $\tilde{Y_2}T_1=T_{34}\tilde{Y_2}, T_1\tilde{X}=\tilde{X}T_{34}$, and $T_1$ is similar to a normal
operator, by Lemma \ref{lem5}, $T_{34}$ is similar to a normal operator, which contracts to Theorem \ref{thm22}.
\end{proof}

\begin{cor}
Assume $T$ is an amenable operator, then $M$ is a maximal invariant subspace such that $T|_M$ is similar to a
normal operator if and only if $M$ is the largest invariant subspace such that $T|_M$ is similar to a normal
operator.
\end{cor}

\begin{cor}
Assume $T$ is an amenable operator and which is quasisimilar to a compact operator, then $T$ is similar to a
normal operator.
\end{cor}

\begin{proof}
Suppose, $TV=VK, WT=KW$ with $V,W$ injective operators having dense ranges and $K$ is a compact operator. Then
$TVKW=VKWT$. Let $C=VKW$, $C\in \mathfrak{A}_{T}'$, and $C$ is a compact operator. According to Theorem
\ref{thm2}, $C$ has the form $\left[\begin{array}{cc}
C_1&\\
&0\\
\end{array}\right]$ respect to the space decomposition in the Theorem. If $Cx=0$, $VWTx=Cx=0$, thus $Tx=0$. It
follows that there is no part of $T_2$, i.e. $T$ is similar to a normal operator.
\end{proof}

\vskip1cm
\section{(Essential) operator valued roots of abelian analytic functions}
In this section, we will study the structure of an operator which is an (essential) operator valued roots of
abelian analytic functions and then
 we get that if  such an operator is also amenable, then it is  similar to a normal operator.
In \cite{G1974} Gilfeather introduce the concept of operator valued roots of abelian analytic functions as
follows: Let $\mathfrak{A}$ is an abelian von Neumann algebra and $\psi(z)$, an $\mathfrak{A}$ valued analytic
function on a domain $\mathcal{D}$ in the complex plane. We may decompose $\mathfrak{A}$ into a direct integral
of factors such that for $A\in\mathfrak{A}$, there exists a unique $g\in L_{\infty}(\wedge, \mu)$ such that
$\mathfrak{A}=\int_\wedge^\oplus g(\lambda)I(\lambda)\mu(d\lambda)$. If $T\in\mathfrak{A}'$ and
$\sigma(T)\subseteq \mathcal{D}$, let
$$\psi(T)=(2\pi i)^{-1}\int_\wedge(T-zI)^{-1}\psi(z)dz.$$
An operator $T$ is called a (essential)roots of the abelian analytic function $\psi$, if $\psi(T)=0$(compact,
respectively). The structure of roots of a locally nonzero abelian analytic function has been given in
\cite{G1974}, in this section we main study the structure of essential roots of a locally nonzero abelian
analytic function.

\begin{lem}\label{lem6}
Assume $T\in\mathfrak{B}(\mathfrak{H})$, $f$ is a locally nonzero analytic function on the neighborhood of
$\sigma(T)$ and assume $f(T)$ is a compact operator, then $T$ is a polynomial compact operator.
\end{lem}
\begin{proof}
Let $\widehat{T}$ denote the image of $T$ in the Calkin algebra, then $\widehat{f(T)}=0$. Since $f$ is a locally
nonzero analytic function on $\sigma(T)$, there exists a polynomial $p$ such that $\widehat{p(T)}=0$. i.e. $T$
is a polynomial compact operator.
\end{proof}

\begin{thm}
Let $\psi$ be a locally nonzero abelian analytic function on $\mathcal{D}$ taking values in the von Neumann
algebra $\mathfrak{A}$. If $T$ is an essential roots of $\psi$ and is amenable, then $T$ is similar to a normal
operator.
\end{thm}

\begin{proof}
Since $\mathfrak{A}$ is an abelian von Neumann algebra, $\mathfrak{A}$ is unitarily equivalent to
$\sum_{n=1}^\infty\oplus \mathfrak{B}_n\oplus\mathfrak{B}_{\infty} $, where $\mathfrak{B}_n$ is a von Neumman
algebra which has
 uniform multiplicity $n$ for all $1\leq n\leq\infty$. Note $T\in \mathfrak{A}'$ is an amenable operator, thus $T=T_1\oplus T_2$,
 where $T_1$ is similar to a normal operator, and $T_1\in(\sum_{n=1}^\infty\oplus \mathfrak{B}_n)',
 T_2\in\mathfrak{B}_{\infty}'$. Let $\sigma_1(\sigma_2)$ denote the continuous  (atom, respectively) parts of the spectrum
  of $\mathfrak{B}_{\infty}$, then $\mathfrak{B}_{\infty}=\mathfrak{C}_{\infty}\oplus\mathfrak{D}_{\infty}$,
  where $\mathfrak{C}_{\infty}$ and $\mathfrak{D}_{\infty}$ are uniform multiplicity $\infty$ von Neumman
algebra and $\sigma(\mathfrak{C}_{\infty})=\sigma_1,\sigma(\mathfrak{D}_{\infty})=\sigma_2$  and $T_2=T_3\oplus
T_4$, where $T_3\in\sigma(\mathfrak{C}_{\infty})', T_4\in\sigma(\mathfrak{D}_{\infty})'$.

Assume $\psi$ is a locally nonzero abelian analytic function on $\mathcal{D}$ and
$\sigma(T)\subseteq\mathcal{D}$,  then $\psi(T)=\psi(T_1)\oplus\psi(T_3)\oplus\psi(T_4)$, note that $\psi(T_3)$
is a compact operator and $\sigma(\mathfrak{C}_{\infty})=\sigma_1$, so $\psi(T_3)=0$. Since
$\mathfrak{D}_{\infty}$ are uniform multiplicity $\infty$ and $\sigma(\mathfrak{D}_{\infty})=\sigma_2$, by Lemma
\ref{lem6}, it follows that $T_4$ is direct sum of polynomial compact operators. According to \cite[Theorem
2.1]{G1974}, there exists a sequence of mutually orthogonal projections $\{P_n, Q_m\}$ in $\mathfrak{A}$ with
$I=\sum P_n+\sum Q_m$ so that $T|_{P_n}$ is finite type spectral operator and $T|_{Q_m}$ is polynomial compact
operator. By \cite[Theorem 3.5, 4.5]{Y}, we get that $T$ is similar to a normal operator.
\end{proof}

\nocite{liyk2,liyk/kua1}

\begin{thebibliography}{4}


\bibitem{A 1977}
 E. A. Azoff;  C. K. Fong and F. Gilfeather,  A reduction theory for
non-self-adjoint operator algebras. Trans. Amer. Math. Soc. 224
(1976) 351--366 (1977).




\bibitem{D1988}
K. R. Davidson, Nest algebras, Longman group UK limited, Essex,
1988.


\bibitem{D1996}
K. R. Davidson, (3-WTRL) $C\sp
*$-algebras by example. (English summary) Fields Institute
Monographs, 6. American Mathematical Society, Providence, RI, 1996.




\bibitem{F1977}
C. K.  Fong, Operator algebras with complemented invariant subspace
lattices. Indiana Univ. Math. J. 26 (1977) 1045--1056.






\bibitem{F2005} D. R. Farenick, B.E. Forrest and L. W. Marcoux,
Amenable operators on Hilbert spaces, J. reine angew. Math.
582(2005) 201-228.

\bibitem{F2007} D. R. Farenick, B.E. Forrest and L. W. Marcoux,
Amenable operators on Hilbert spaces, J. reine angew. Math.
602(2007) 235.

\bibitem{G2006}
J. A. Gifford, Operator algebras with a reduction proprety, J. Aust.
Math. Soc 80 (2006) 279--315.

\bibitem{G1974} F. Gilfeather,  Operator valued roots of abelian analytic functions,
 Pacific J. Math. 55 (1974), 127--148.

\bibitem{H1978}
D. W. Hadwin,  An asymptotic double commutant theorem for
$C\sp{\ast} $-algebras. Trans. Amer. Math. Soc. 244 (1978) 273--297.

\bibitem{BE1972}
B. E. Johnson. Cohomology in Banach Algebras. Mem. Amer. Math. Soc.
Vol. 127 (Amer. Math. Soc., 1972).


\bibitem{R1969} R. G. Douglas, On operators similar to normal
operators, Rev. Roum. Math. Pures Appl. 14(1969) 193-197.

\bibitem{R1982} S. Rosenoer, Completely reducible operator algebras and spectral synthesis,
 Canad. J. Math. 34 (1982), no. 5, 1025--1035.

\bibitem{R1993} S. Rosenoer,  Completely reducible algebras containing compact operators,
 J. Operator Theory. 29 (1993), no. 2, 269--285.

\bibitem{R2003}
H. Radjavi and P. Rosenthal, Invariant subspaces. Second edition.
Dover Publications, Inc., Mineola, NY, 2003.



\bibitem{S1967}
J. T. Schwartz, W*-algebras, Gordon and Breach, New York, 1967.






 \bibitem{W1995}
G. A. Willis, When the algebra generated by an operator is amenable,
J. Operator Theorey. 34 (1995) 239--249.

\bibitem{Y}
Y. Q. Ji and L. Y. Shi, Amenable operators on Hilbert spaces,
 Houston Journal of Mathematics(to appear).




\end{thebibliography}

\end{document}